\documentclass[reqno]{amsart}

\usepackage{multicol}
\usepackage{fullpage}
\usepackage{bbm,color}
\usepackage{graphics,array}
\usepackage{tikz}
\usepackage{ytableau}
\usetikzlibrary{matrix,arrows,snakes,patterns,cd,shapes,graphs,graphs.standard}
\usepackage{amsmath,amsthm,mathtools,hyperref}
\usepackage{amsfonts}
\usepackage{amssymb}
\definecolor{light-gray}{gray}{0.75}
\definecolor{ww}{gray}{0.9}
\definecolor{gg}{gray}{0.5}
\definecolor{bb}{gray}{0.1}

\usepackage{soul}  

\newtheorem{theorem}{Theorem}[section]
 \newtheorem{corollary}[theorem]{Corollary}
 \newtheorem{prop}[theorem]{Proposition}

 \theoremstyle{definition}
 \newtheorem{example}[theorem]{Example}
 \newtheorem{question}[theorem]{Question}
 \newtheorem{remark}[theorem]{Remark}
 
 \newtheorem{definition}[theorem]{Definition}
 \newtheorem{defn}[theorem]{Definition}
 
\newcommand{\red}[1]{\textcolor{red}{#1}}

\definecolor{myblue}{RGB}{80,80,160}
\definecolor{mygreen}{RGB}{80,160,80}

\newcommand{\GETOUT}[1]{}
\newcommand{\redst}[1]{}

\newcommand{\thresh}{\mathrm{deg}}

\newcommand{\Motzkin}[2]{\mathsf{Motzkin}_{#1,#2}}
\newcommand{\wf}{\mathsf{f}}
\newcommand{\wg}{\mathsf{g}}

\newcommand{\splitgraphicnew}{
\begin{tikzpicture}[
my node style/.style={circle,draw,inner sep=1pt}
]
  \node[my node style]  (v0) at (0,1) {$v_1$};
  \node[my node style]  (v1) at (.951,.309)  {$v_2$};
  \node[my node style]  (v4) at (-.951,.309)  {$v_5$};
  \node[my node style]  (v2) at (.5878,-.809)  {$v_3$};
  \node[my node style]  (v3) at (-.5878,-.809)  {$v_4$};
  \node[my node style]  (v5) at (4,1.5) {$w_1$};
  \node[my node style]  (v6) at (4,0.5) {$w_2$};
  \node[my node style]  (v7) at (4,-0.5) {$w_3$};
  \node[my node style]  (v8) at (4,-1.5) {$w_4$};
  \foreach \i/\j in {0/1,0/2,0/3,0/4,1/2,1/3,1/4,2/3,2/4,3/4,0/7,1/8,2/8,4/8}
    \path (v\i) edge (v\j);
  \path (v0) edge [bend right=0] (v5);
  \path (v0) edge [bend right=0] (v6);
  \path (v0) edge [bend left=10] (v7);
  \path (v1) edge [bend right=0] (v5);
  \path (v1) edge [bend right=0] (v6);
  \path (v1) edge [bend left=0] (v7);
  \path (v2) edge [bend right=0] (v5);
  \path (v2) edge [bend right=0] (v6);
  \path (v2) edge [bend left=0] (v7);
  \path (v3) edge [bend right=0] (v5);
  \path (v3) edge [bend left=5] (v6);
  \path (v3) edge [bend right=20] (v7);
  \path (v4) edge [bend left=2] (v5);
  \path (v4) edge [bend left=9] (v6);
  \path (v4) edge [bend right=0] (v7);
  \path (v3) edge [bend right=10] (v8);
  \path (v0) edge [bend left=15] (v8);
\end{tikzpicture}
}

\newcommand{\motzkinpath}{
\begin{tikzpicture}[scale=0.5],
my node style/.style={circle,draw,inner sep=1pt}
]
  \draw (0,0) -- ++(1,0) -- ++(1,1) -- ++(1,0) -- ++(1,0) -- ++(1,1) -- ++(1,-1) -- ++(1,0) -- ++(1,1) -- ++(1,-1) -- ++(1,-1) ;
  \filldraw [black] (0,0) circle (2pt)
					(1,0) circle (2pt)
					(2,1) circle (2pt) 
					(3,1) circle (2pt) 
					(4,1) circle (2pt) 
					(5,2) circle (2pt) 
					(6,1) circle (2pt) 
					(7,1) circle (2pt) 
					(8,2) circle (2pt) 
					(9,1) circle (2pt) 
					(10,0) circle (2pt);
\end{tikzpicture}
}

\newcommand{\spanningtree}{
\begin{tikzpicture}[
my node style/.style={circle,draw,inner sep=1pt}
]
  \node[my node style]  (v1) at (0,1) {$v_1$};
  \node[my node style]  (v2) at (0,0)  {$v_2$};
  \node[my node style]  (v3) at (2,0.5)  {$v_3$};
  \node[my node style]  (v4) at (1,-1)  {$v_4$};
  \node[my node style]  (v5) at (1,1)  {$v_5$};
  \node[my node style]  (v6) at (3,0.5)  {$v_6$};
  \node[my node style]  (v7) at (1,0)  {$v_7$};
  \node[my node style]  (v8) at (2,2)  {$v_8$};
  \node[my node style]  (v9) at (2,-2)  {$v_9$};
  \foreach \i/\j in {1/5,5/8,5/3,3/6,3/7,7/2,7/4,4/9}
    \path (v\i) edge (v\j);
\end{tikzpicture}
}

\def\Split{S}
\def\Sandpile{\mathsf{ASM}}
\def\Rec{\mathsf{Rec}}

\newcommand{\car}[1]{\fbox{#1}}

\title{The sandpile model on the complete split graph,\\ Motzkin words, and tiered parking functions}
\author{Mark Dukes}
\address{School of Mathematics and Statistics, University College Dublin, Ireland.}
\email{mark.dukes@ucd.ie}
\date{\today}
\begin{document}
\maketitle
\begin{abstract}
We classify recurrent states of the Abelian sandpile model (ASM) on the complete split graph.
There are two distinct cases to be considered that depend upon the location of the sink vertex in the complete split graph.
This characterisation of decreasing recurrent states is in terms of Motzkin words and can also be characterised in terms of combinatorial necklaces.
We also give a characterisation of the recurrent states in terms of a new type of parking function that we call a {\it{tiered parking function}}.
These parking functions are characterised by assigning a tier (or colour) to each of the cars, and specifying how many cars of a lower-tier one wishes to have parked before them.
We also enumerate the different sets of recurrent configurations studied in this paper, and in doing so derive a formula for the number of spanning trees of the complete split graph that uses a bijective Pr\"ufer code argument.
\end{abstract}

\section{Introduction}

The complete split graph is a bipartite graph consisting of two distinct parts, a clique part in which all distinct pairs of vertices are connected by a single edge, and an independent part in which no two vertices are connected to an edge. There is precisely one edge between every vertex in the clique part and every vertex in the independent part. 
We denote the complete split graph which has vertices $\{v_1,\ldots,v_m\}$ in the clique part and vertices $\{w_1,\ldots,w_n\}$ in the independent part by $\Split_{m,n}$. The graph $S_{5,4}$ is illustrated in Figure~\ref{figsplit}.

\begin{figure}[!h]
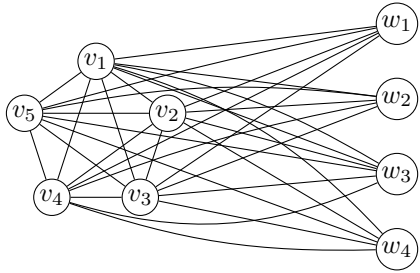

\label{figsplit}
	\begin{center}
	\splitgraphicnew
	\end{center}
\caption{The complete split graph $S_{5,4}$
}
\end{figure}

The graph $S_{m,n}$ contains the complete graph $K_m$ as a subgraph, but is also a bipartite graph in its own right, and it is this dual feature that we find interesting to examine in terms of the sandpile model.
The sandpile model has been studied on several classes of graphs, and rich connections to other combinatorial structures have been established in each of the cases. 
For the case of the sandpile model on the complete graph $K_{n+1}$ in which one designated vertex is a sink, it was shown by Cori and Rossin~\cite{corikn} that the set of recurrent states of the sandpile model on this graph are in one-to-one correspondence with parking functions of order $n$.

In the case of the sandpile model on the complete bipartite graph $K_{m,n}$ in which a designated vertex is the sink, it was shown by the author (in collaboration with Le Borgne~\cite{dlb}) that the set of recurrent states admit a description in terms of planar animals called {\it{parallelogram polyominoes}}.
Parallelogram polyominoes are also known as {\it{staircase polyominoes}} in the literature.
Cori and Poulalhon in \cite{coripoul} showed that the recurrent states of the sandpile model on the complete tri-partite graph $K_{1,p,q}$, in which the solitary vertex is the sink, admits a description in terms of a parking function for cars of two different colours.

In this paper we will classify the recurrent states of the Abelian sandpile model on the complete split graph.
There are two distinct cases to be considered that depend on whether the sink is a clique vertex or an independent vertex in the graph.
The classification is in terms of two different sets of Motzkin words, 
and we explain how these words can also be considered as unique representations of certain 3-coloured combinatorial necklaces.

Further to this, we give a second characterisation of the recurrent states in terms of a new type of parking function that we call a {\it{tiered parking function}}.
These parking functions are characterised by assigning a tier (or colour) to each of the cars, and specifying how many cars of a lower-tier one wishes to have parked before them in a one-way street.
We also enumerate the different sets of recurrent configurations studied in this paper, and in doing so derive a formula for the number of spanning trees of the complete split graph.
All the proofs given in this paper are bijective ones.

This paper lays the foundations for a study into statistics on these recurrent configurations, as was done in the case of the 
complete bipartite graph in \cite{dlb} and \cite{aadhl,aadl}.
We posit that there are many interesting correspondences to be uncovered by studying  
Dhar's burning algorithm (that characterises recurrent configurations)
in the context of bijective combinatorics, as evidenced by \cite{dukesselig,ferrers,permgraphs,sss}

\section{Recurrent states of the sandpile model on a graph}

The Abelian sandpile model (ASM) may be defined on any undirected graph $G$ with a designated vertex~$s$ called the \emph{sink}.  
A \emph{configuration} on $G$ is an assignment of non-negative integers to the non-sink vertices of $G$,
$$c:\mathrm{Vertices}(G)\backslash\{s\} ~\mapsto ~ \mathbb{N}=\{0,1,2,\ldots\}.$$
The number $c(v)$ is sometimes referred to as the number of \emph{grains} at vertex~$v$, or as the \emph{height} of $v$.  
Given a configuration $c$, a vertex $v$ is said to be {\emph{stable}} if the number of grains at $v$ is strictly smaller than the threshold of that vertex, which is the degree of $v$, denoted $\thresh(v)$.
Otherwise $v$ is {\emph{unstable}}.  A configuration is called {\emph{stable}} if all non-sink vertices are stable.

If a vertex is unstable then it may {\emph{topple}}, which means the vertex donates one grain to each of its neighbors.  
The sink vertex has no height associated with it and it only absorbs grains, thereby modelling grains exiting the system.
Given this, it is possible to show 
that starting from any configuration $c$ and toppling unstable vertices, one eventually reaches a stable configuration $c'$.  
Moreover, $c'$ does not depend on the order in which vertices are toppled in this sequence. We call $c'$ the \emph{stabilisation} of $c$.
We use the notation $\Sandpile(G,s)$ to indicate that we are considering the Abelian sandpile model as described here on the graph $G$ with sink $s$.

Starting from the empty configuration, one may indefinitely add any number of grains to any vertices in $G$ and topple vertices should they become unstable.  
Certain stable configurations will appear again and again, that is, they \emph{recur}, while other stable configurations will never appear again.
These {\emph{recurrent configurations}} are the ones that appear in the long term limit of the system. 
Determining the set of recurrent configurations for $\Sandpile(G,s)$ is not a straightforward task. In \cite[Section 6]{Dhar}, Dhar describes the so-called \emph{burning algorithm}, which establishes in linear time whether a given stable configuration is recurrent.
We recall the result here.

\begin{prop}[\cite{Dhar}, Section 6.1]\label{pro:DharBurning}
Let $G$ be a graph with sink $s$.
A stable configuration $c$ on $G$ is recurrent if and only if there exists an ordering $v_0=s,v_1,\ldots,v_{n}$ of the vertices of $G$ such that, starting from $c$, for any $i \geq 1$, toppling the vertices $v_0,\ldots,v_{i-1}$ causes the vertex $v_i$ to become unstable. 
\end{prop}

This algorithm has seen several equivalent formulations over the years, and we refer the interested reader to two recent books on the topic that discuss these in a clear and insightful way; Corry \& Perkinson~\cite[\S7.5]{perkinson} and Klivans~\cite[\S2.6.7]{klivans}.

A configuration on $\Sandpile(\Split_{m,n},s)$ is a vector $c=(c_1,\ldots,c_{i-1},-,c_{i+1},\ldots,c_{m+n-1})$ 
whereby $c_i$ represents the number of chips/grains at the $i$th vertex in the 
sequence $(v_1,\ldots,v_m,w_1,\ldots,w_n)$. The dash indicates the location corresponding to the sink.
We will use a semi-colon in the configurations to distinguish between the clique and independent parts, e.g. 
$c=(a_1,\ldots,a_{m-1},-;b_1,\ldots,b_n)$.
Let $\Rec(\Sandpile(\Split_{m,n},s))$ be the set of recurrent states of the model.
We will restrict our analysis to those configurations that are weakly decreasing with respect to vertex labels. 
This restriction will allow us to focus our analysis to consider characterizing all those `different' configurations, and from which we can generate all configurations through permutations.
The two sets we will consider are the weakly decreasing configurations of the sets $\Sandpile(\Split_{m,n},v_m)$ and $\Sandpile(\Split_{m,n},w_n)$.

\section{Motzkin words and recurrent states of the split graph}
\newcommand{\mw}{\alpha}
\newcommand{\mtc}{f}
\subsection{Recurrent states for the case of a clique-sink}
We will use Motzkin words for the characterisation of recurrent states for both cases.
A {\em{Motzkin path $P$ of length $p$ }} is a lattice path in the plane from $(0,0)$ to $(p,0)$ which never passes below the
$x$-axis and whose permitted steps are the up step $U=(1,1)$, the down step $D=(1,-1)$, and the horizontal step $H=(1,0)$.
An example of a Motzkin path is given in Figure~\ref{motzexamp}. 
The Motzkin word of a path is a listing of the $p$ steps of the path in the order they appear from left to right,
e.g.  the Motzkin word of Figure~\ref{motzexamp} is $\mw=HUHHUDHUDD$.

\begin{figure}
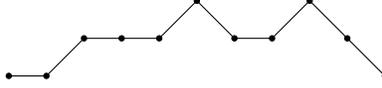

\begin{center}
\motzkinpath
\end{center}
\caption{Example of a Motzkin path of length 10\label{motzexamp}.}
\end{figure}

\begin{defn} \label{motzkindefn}
A {\em Motzkin word} is a word $\mw$ consisting of the letters $U$, $D$, $H$ with the properties (i) in counting $\mw$ from left to right
the $U$ count is always greater than or equal to the $D$ count, and (ii) the total $U$ count is equal to the $D$ count.
Let $\Motzkin{m}{n}$ be the set of Motzkin words with $m$ occurrences of $U$  and $n$ occurrences of $H$.
\end{defn}

The deletion of the $H$ letters in a Motzkin word gives a Dyck word and, conversely, a Motzkin word with $m$ occurrences of $U$ and 
$n$ occurrences of $H$ is obtained by taking a Dyck word with $m$ occurrences of $U$ and inserting $n$ occurrences of $H$ in 
any positions in this word. This decomposition provides an enumeration of such words:
$$|\Motzkin{m}{n}| = \dfrac{1}{m+1} {2m \choose m} {2m+n \choose n}.$$
\begin{defn}
To any Motzkin word $\mw \in \Motzkin{m-1}{n}$ we associate a configuration 
$$\mtc(\mw)= (a_1,\ldots,a_{m-1},-; b_1,b_2,\ldots,b_n)$$ on $\Split_{m,n}$ as follows:
\begin{itemize}
\item $a_i$ is such that $a_i+1$ is equal to the number of occurrences of the letter $D$ plus the number of occurrences of the letter $H$ appearing in $\mw$ after the $i$-th occurrence of the letter $U$.
\item $b_i$ is equal to the number of occurrences of the letter $D$ appearing after the $i$-th occurrence of the letter $H$ in $\mw$.
\end{itemize}
\end{defn}
Observe that, by construction, the sequences $(a_1,\ldots,a_{m-1})$ and $(b_1,\ldots,b_n)$ are weakly decreasing.
For the Motzkin word example above, the associated configuration on $\Sandpile(\Split_{4,4},v_4)$ is 
$$\mtc(\mw)=(5,3,1,-;3,3,3,2).$$

\begin{theorem}\label{thmone}
A weakly decreasing stable configuration $c=(a_1,\ldots,a_{m-1},-;b_1,\ldots,b_n)$ on $\Sandpile(\Split_{m,n},v_m)$ is recurrent 
iff it corresponds, via the construction $f$, to a unique Motzkin word $\mw$ in $\Motzkin{m-1}{n}$.
\end{theorem}

\begin{proof}
Suppose that $\mw$ is a Motzkin word in $\Motzkin{m-1}{n}$.
The positions of the occurrences of the letters $U$ and $H$ in $\mw$ give the toppling order required in Dhar's toppling order (Proposition~\ref{pro:DharBurning}).
The $i$-th occurrence of $U$ corresponds to the toppling of vertex $v_i$ and the $j$-th occurrence of $H$ corresponds to the toppling of vertex $w_j$. 
In the example above with $\mw=HUHHUDHUDD$, the toppling order is 
$$w_1,v_1,w_2,w_3,v_2,w_4,v_3.$$
We make the following observations using the definition for $\mtc(\alpha)$:
\begin{itemize}
\item 
	As $v_i$ corresponds to the $i$-th occurrence of $U$ in $\mw$, and the height $a_i$ associated with it is such that $a_i+1$ is the number of $D$'s and $U$'s to its right, we must have $a_i+1 \leq m-1+n$, i.e. $v_i$ is stable. 
	Similarly, as $w_i$ corresponds to the $i$-th occurrence of $H$ in $\mw$, and the height $b_i$ associated with it is such that $b_i$ is the number of $D$'s to its right, we must have $b_i \leq m-1$, the maximum number of such $D$'s. This shows that $w_i$ must be stable.
\item 
	Suppose that, as in Dhar's toppling algorithm, we have toppled the sink and the $i-1$ vertices that correspond to the first non-$D$ symbols of $\mw$.
	Toppling the sink will contribute one extra grain to every vertex. The $i$-th non-$D$ symbols of $\mw$ can be $H$ or $U$.

	If the $i$-th non-$D$ symbol of $\mw$ is $H$ then this $H$ corresponds to the vertex $w_k$ for some $k$. 
	To the left of this $H$ in $\mw$ there are $k-1$ $H$'s and $i-k$ $U$'s. 
	As each of the these $i-k$ $U$'s have already toppled, they will have added $i-k$ grains to the initial height of $w_k$, and topplings of the vertices corresponding to $H$'s will have no effect on this height.
	This means the current height of $w_k$ is $b_k+1+(i-k)$. 
	The value $b_k+1+(i-k)$ equals the number of occurrences of $D$ in $\mw$ after this $i$-th $H$ we are considering, plus one, 
  	plus the number of occurrences of $U$ to its left.
	As the number of $U$'s to its left is always at least as large as the number of $D$'s to its left, 
	it must be the case that the current height of $w_k$ is at least the total number of $D$'s plus one, 
	which equals $m$, and $w_k$ is therefore unstable. 

	If, instead, the $i$-th non-$D$ symbol of $\mw$ is $U$ then this current $U$ corresponds to the vertex $v_k$ for some $k$. 
	To the left of this $U$ in $\mw$ there are $i-1$ non-$D$ symbols (each either $U$ or $H$) representing vertices that toppled and each added 1 grain to the height of vertex $v_k$.
	This means the current height, $h$ say, of $v_k$ is $a_k$ (the initial height) plus 1 (from the toppling of the sink) + $(i-1)$ from the topplings just mentioned, i.e. $h=a_k+1+(i-1)$. 
	By definition, $a_k$ such that $a_k+1$ $=$ the number of $D$'s and $H$'s appearing after the `current $U$'.
	 Also, the number of $U$'s and the number of $H$'s to the left of this 'current $U$' is at least as large as the number of $D$'s and the number of $H$'s to the left of this 'current $U$'.
	Therefore we have $h=a_k+1+(i-1) \geq 1+\mbox{number of $D$'s and $H$'s in $\mw$} = 1+m+n-1 = m+n$.
	This implies that vertex $v_k$ is unstable.
\end{itemize}
Therefore, by Dhar's burning algorithm, $\mtc(\alpha)$ will be a recurrent configuration that is, by construction, weakly decreasing.

Next, suppose we have a recurrent configuration $c=(a_1,\ldots,a_{m-1},-;b_1,\ldots,b_n)$ that is weakly decreasing.
Construct a word $\beta$ as follows. Let $\beta$ initially consist of only $n$ $H$'s.
\begin{itemize}
\item Insert $m-1$ $D$'s into $\beta$ in such a way that there are $b_i$ occurrences of the symbol $D$ appearing after the $i$-th occurrence of the symbol $H$. 
\item Next insert $m-1$ $U$'s so that there are $a_i+1$ occurrences of $D$ or $H$ appearing after the $i$-th occurrence of $U$ in $\beta$.
\end{itemize}
This construction will always have the property that the number of $U$'s seen in any prefix is at least as large as the number of $D$'s in that prefix, i.e., 
the reduced word (without the $H$'s) is a Dyck word, and so (by Definition~\ref{motzkindefn}) the word $\beta$ is a Motzkin word consisting of $m-1$ $U$'s and $n$ $H$'s.
\end{proof}

As the number of such decreasing recurrent configurations is equal to the number $|\Motzkin{m-1}{n}|$, we have the following:

\begin{corollary}
The number of weakly decreasing recurrent configurations on $\Split_{m,n}$ with sink $v_m$ is equal to
$$\dfrac{1}{m} {2m-2 \choose m-1} {2m-2+n \choose n}.$$
\end{corollary}

\subsection{Recurrent states for the case of an independent-sink}
We have a similar characterisation as in the previous case, however we have to restrict the set of Motzkin words to those for which the first occurrence of $H$ appears after the first occurrence of $D$. We call these words {\it{$DH$-Motzkin words}}.

\begin{defn}\label{gdef}
To any $DH$-Motzkin word $\mw$ having $m$ occurrences of the letter $U$ and $n-1$ occurrences of the letter $H$ we associate 
a configuration $g(\mw) = (a_1,\ldots,a_m;b_1,\ldots,b_{n-1},-)$ on $\Sandpile(\Split_{m,n},w_n)$ as follows:
\begin{itemize}
\item $a_i$ is such that $a_i+1$ is equal to the number of occurrences of the letters $D$ and $H$ in $\mw$ that appear after the 
$i$-th occurrence of the letter $U$ in $\mw$.
\item $b_i$ is equal to the number of occurrences of the letter $D$ appearing after the $i$-th occurrence of the letter $H$ in $\mw$.
\end{itemize}
\end{defn}

For example, to the $DH$-Motzkin word $$\mw = UUDHUDHUDDH$$
is associated the configuration $g(\mw)=(6,6,4,2;3,2,0,-)$ on $\Sandpile(\Split_{4,4},w_4)$.

\begin{theorem}\label{thmtwo}
A weakly decreasing stable configuration $c=(a_1,\ldots,a_{m};b_1,\ldots,b_{n-1},-)$ on $\Sandpile(\Split_{m,n},w_n)$ is recurrent 
iff it corresponds, via the construction $g$, to a unique $DH$-Motzkin word $\mw$ in $\Motzkin{m}{n-1}$.
\end{theorem}

\begin{proof}
This proof largely mirrors the proof of Theorem~\ref{thmone} so we omit most of the details. 
The place in which the proofs differ is considering the first occurrence of $H$, which corresponds to the vertex $w_1$.
In checking Dhar's burning algorithm on a configuration, the sink $w_n$ topples first and increases the heights of all clique vertices $\{v_1,\ldots,v_m\}$ by 1. 
At this time all independent vertices are still stable as none have increased in height as a result of toppling the sink.
According to the construction $g$ of Definition~\ref{gdef} there must then be at least one $D$ to the left of this first $H$ to ensure stability.
\end{proof}

In order to obtain the number of $DH$-Motzkin words in $\Motzkin{m}{n-1}$, we will give a bijection with a family 
of standard Young tableaux that is inspired by a paper of Stanley~\cite{stanley_syt}.

\begin{prop}
The number of weakly decreasing recurrent configurations in $\Sandpile(\Split_{m,n},w_n)$ is
$$\dfrac{1}{m} {2m+n-1 \choose m-1} {m+n-2 \choose m-1}.$$
\end{prop}

\begin{proof}
By Theorem~\ref{thmtwo}, this number is equal to 
the number of $DH$-Motzkin words in $\Motzkin{m}{n-1}$ is
Consider the Ferrers diagram $F$ of shape $(m,m,\stackrel{n-1}{\overbrace{1,1,\ldots,1}})$.
To a $DH$-Motzkin word $\mw$ in $\Motzkin{m}{n-1}$ we associate a standard Young tableau $T$ of shape $F$ in the following way.
Let the first row of $T$ contain the positions of the $U$'s in the Motzkin word $\mw$ and
let the second row of $T$ contain the positions of the $D$'s in $\mw$.
Fill the $n-1$ remaining cells in the first column with the positions of the $H$'s in $\mw$. 
For example, the standard Young tableau for the word $\mw = UUDHUDHUDDH$ is illustrated in Figure~\ref{sytab}.

\begin{figure}
\begin{center}
\begin{tikzpicture}
\def\step{0.43}
\begin{scope}[xshift=0cm, yshift=0cm]
\draw (0,0) node [anchor = north west]  {\ytableausetup{centertableaux}\begin{ytableau} 1 & 2 & 5 & 8 \\ 3 & 6 & 9 & 10 \\ 4 \\ 7 \\ 11 \end{ytableau}};
\draw[left] (-0.0,-1.5) node {$T = $};
\end{scope}
\end{tikzpicture}
\end{center}
\caption{The standard Young tableau corresponding to the word $\mw = UUDHUDHUDDH$ in the proof of Theorem~\ref{thmtwo}.\label{sytab}}
\end{figure}

Every $DH$-Motzkin word in $\Motzkin{m}{n-1}$ can be represented as a standard Young tableau of shape $(m,m,\stackrel{n-1}{\overbrace{1,1,\ldots,1}})$, and vice-versa. 
The condition that the first occurrence of $H$ appears after the first occurrence of $D$ ensures that the values are increasing down the first column.
The number of standard Young tableaux having shape $(m,m,\stackrel{n-1}{\overbrace{1,1,\ldots,1}})$ is given by the hook length formula. 
The hook lengths of the cells in the first row of $F$ (from left to right) are $n+m$, $m$, $m-1$, $\ldots$, 2.
The hook lengths of the cells in the second row (from left to right) are $n+m-1$, $m-1$, $m-2$, $\ldots$, 1.
The hook lengths for the remaining cells in the first column (from top to bottom) are $n-1,n-2,\ldots,1$. 
Since the total number of cells is $2m+n-1$, the number of such tableaux is 
$$\dfrac{(2m+n-1)!}{(n+m)(n+m-1)m!(m-1)!(n-1)!},$$
which can be written $\frac{1}{m} {2m+n-1 \choose m-1} {m+n-2 \choose m-1}.$
\end{proof}

\subsection{A Pr\"ufer code decomposition for spanning trees of the complete split graph}

It is a well established fact the the number of recurrent states of the Abelian sandpile model on a graph is equal 
to the number of spanning trees of that graph.
In this section we will enumerate the set of recurrent states by enumerating the spanning trees of the complete split graph.
We will do this by presenting a bijective proof that uses the Pr\"ufer code of the spanning trees.

In this subsection we will use a different labelling of the vertices of the complete split graph.
Let $\Split_{m,n}$ have clique vertices $\{v_1,v_2,\ldots,v_{m}\}$ and independent vertices $\{v_{m+1},\ldots,v_{m+n}\}$. 
Suppose these vertex labels are totally ordered:
$$v_1<v_2<\cdots < v_{m} < v_{m+1}< \cdots < v_{m+n}.$$
Any spanning tree $T$ of $\Split_{m,n}$ is a tree with vertices labelled $v_1,\ldots,v_{m+n}$ such that
for any edge $(v_i,v_j)$ of $T$, one has $i\leq m$ or $j \leq m$.

The Pr\"ufer code of this tree is obtained by successively deleting the leaves of $T$ with minimal label and recording the vertex to which they were attached, until $T$ has only one edge.
The unique edge that is obtained at the end of the procedure is $(v_{m+n},v_i)$ where $i\leq m$. 
Two words $\wf,\wg$ may be built using the alphabet consisting of the labels of the vertices.
Initially both $\wf$ and $\wg$ are empty, 
and at each step of the procedure a letter is added either to $\wf$ or to $\wg$.
It is added to $\wf$ if the leaf deleted is a clique vertex and added to $\wg$ if it is an independent vertex. 
The letter added is the label of the vertex neighbour of the deleted leaf.

\begin{figure}
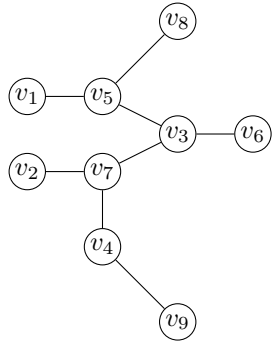

\begin{center}
\spanningtree
\end{center}
\caption{A spanning tree of the complete split graph $\Split_{5,4}$.\label{stexample}}
\end{figure}

For the tree given in Figure~\ref{stexample}, the vertices are deleted in the following order
$$v_1,v_2,v_6,v_8,v_5,v_3,v_7.$$
The Pr\"ufer code of the tree is $v_5,v_7,v_3,v_5,v_3,v_7,v_4$ and the construction 
of the two words $\wf$ and $\wg$ gives
$$\wf=v_5,v_7,v_3,v_7 \qquad \qquad \wg=v_3,v_5,v_4.$$

\begin{prop}\label{threeeight}
There is a bijection between spanning trees of $\Split_{m,n}$ and pairs of words $(\wf,\wg)$ such that
$\wf$ has length $m-1$ and has letters in the alphabet $\{v_1,\ldots,v_{m+n}\}$, while $\wg$ has $n-1$ letters 
in the alphabet $\{v_1,\ldots,v_m\}$.
\end{prop}

\begin{proof}
Since at the end of the coding there remains an edge connecting $v_{m+n}$ and $v_i$ where $i\leq m$
there were $m-1$ clique vertices deleted to give the word $\wf$ and $n-1$ independent vertices
deleted to give the word $\wg$.
The neighbours of the independent vertices are $v_1,v_2,\ldots,v_m$ while the neighbours of the clique vertices
are $v_1,\ldots,v_{m+n}$.

In order to rebuild the tree one has to notice that the leaves of the coded tree are those vertices not appearing in $\wf$ nor in $\wg$.
The edges of the tree are obtained iteratively by determining the smallest leaf $v_i$, and then adding an edge from it to the first letter
of $\wf$ if $i\leq m$ or to the first letter of $\wg$ if $i>m$. 
Then proceed by iteratively deleting the first letter of $\wf$ or $\wg$, accordingly. 
In this iteration a vertex becomes a leaf if it does not appear in the remaining letters of the words $\wf$ and $\wg$.
\end{proof}

\begin{corollary}
The number of spanning trees of $\Split_{m,n}$ is $(m+n)^{m-1} m^{n-1}$ and, consequently, this is the number of recurrent configurations of the Abelian sandpile model on the complete split graph $\Split_{m,n}$.
\end{corollary}

\begin{proof}
By Proposition~\ref{threeeight}, this is the number of pairs of words $(\wf,\wg)$ such that 
$\wf$ is a word of length $m-1$ having letters in the alphabet $\{v_1,\ldots,v_{m+n}\}$, of which there are $m^{n-1}$,
and $\wg$ is the number of words of length $n-1$ having letters in the alphabet $\{v_1,\ldots,v_{m}\}$, of which there are $(m+n)^{m-1}$.
\end{proof}

\subsection{Combinatorial necklaces}
We may offer another characterisation of the decreasing recurrent states of the Abelian sandpile model by using the correspondence with Motzkin words that we have established. A combinatorial necklace is an arrangement of coloured beads in a circle. Two necklaces are called {\it{equivalent}} if one is a rotation of the other and otherwise they are {\it{different}}. 
The colours of a combinatorial necklace are from some alphabet, typically the first $k$ integers.

\newcommand{\NecklacesOne}{\mathsf{Necklaces}_1}
Let us define $\NecklacesOne(a,b,c)$ to be the set of all different 3-coloured combinatorial necklaces on the alphabet $\{U,D,H\}$ that contain 
$a$ beads of `colour' $U$, $b$ beads of colour $D$, and $c$ beads of colour $H$.
Then the set $\NecklacesOne(m,m-1,n)$ is in one-to-one correspondence with the set of weakly decreasing $\Rec(\Sandpile(\Split_{m,n};v_m))$ via the following association:
if $\mw$ is the Motzkin word that corresponds to the weakly decreasing $c \in \Rec(\Sandpile(\Split_{m,n};v_m))$, then
the necklace that corresponds to $c$ is the one whose clockwise reading is $U\mw$, the Motzkin word $\mw$ with $U$ prepended.

The class of combinatorial necklaces  $\NecklacesOne(m,m-1,n)$ is equivalent to the set of Motzkin words $\Motzkin{m-1}{n}$ so it is possible to present and prove the characterisation of the recurrent states using necklaces instead. 

Notice that, by virtue of this construction, a necklace in $\NecklacesOne(m,m-1,n)$ cannot have two different representations.
The unique representation is achieved by considering the sub-necklace consisting of $U$'s and $D$'s. As there will be one more $U$ than $D$, determine the position of $U$ such that the $U$'s and $D$'s that follow it are a Dyck word. 
Delete this $U$ and read the remainder of the necklace in a clockwise direction. 
Let $\mw$ be the sequence that results. This $\mw$ will be a Motzkin word. 
Reversing this procedure is straightforward as we simply prepend a $U$ to the word and write the symbols as they appear (clockwise) in a circle.
The combinatorial necklace that corresponds to the Motzkin word $\mw= HUHHUDHUDD \in \Motzkin{3}{4}$ is illustrated in Figure~\ref{cnexample}.
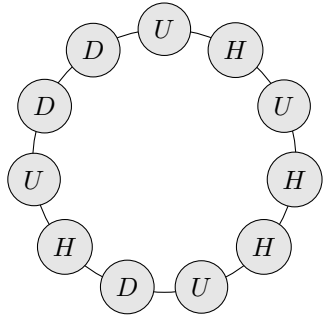
\begin{figure}
\begin{center}
\begin{tikzpicture}
\def\nnumber{11}
\def\radius{1.75}
\draw (0,0) circle (\radius);
\foreach \r/\l in {0/U,1/H,2/U,3/H,4/H,5/U,6/D,7/H,8/U,9/D,10/D}
	\node [circle,draw,fill=ww] () at ({90-(360/\nnumber)*(\r)}:\radius) {$\l$};
\end{tikzpicture}
\end{center}
\caption{The combinatorial necklace in $\NecklacesOne(4,3,4)$ that corresponds to $\mw= HUHHUDHUDD$. The top $U$ is the `special' $U$ signifying the start of the associated word.\label{cnexample}}
\end{figure}

We can do the same for those $DH$-Motzkin words in $\Motzkin{m}{n-1}$  that correspond to weakly decreasing configurations in $\Rec(\Sandpile(\Split_{m,n};w_n))$. Use the same procedure of prepending a $U$ to the $DH$-Motzkin word so as to be able to identify a starting point of the necklace. The resulting collection of necklaces is a strict subset of $\NecklacesOne(m+1,m,n-1)$.

\section{Tiered parking functions}

In light of the classification of weakly decreasing recurrent states as Motzkin words/combinatorial necklaces, we may also offer the following equivalent definition 
of such configurations as a new type of parking function that we will call a {\em tiered parking function}. 
Parking functions were mentioned earlier in the paper in relation to the recurrent states of the sandpile model on the complete graph.
Moreover, the $G$-parking functions of Postnikov and Shapiro~\cite{postshap} provide a useful language in which an alternative description of recurrent states of the sandpile model on a general graph may be given.

The application of $G$-parking functions to the complete split graph is different to what we present in this section. 
Our aim is to provide a new `type' of parking function, and provide a setting in which recurrence is quite easily established in this new context.
We refer the interested reader to Yan~\cite{yan} for a discussion of $G$-parking functions and their relation to the ASM.
Our definition is inspired by Cori and Poulalhon's~\cite{coripoul} concept of $(p,q)$-parking functions.

\begin{definition}[$k$-tiered parking function]
\label{tpf}
Let $m_1,\ldots,m_k$ be a sequence of positive integers with $m_1+\ldots+m_k=M$. 
Suppose that there are $m_i$ cars of colour/tier $i$ and there are $M$ parking spaces.
We will call a sequence $P=(m_1;P_2,\ldots,P_k)$ of sequences $P_i = (p^{(i)}_1,\ldots,p^{(i)}_{m_i})$ a $k$-tiered parking function of order $(m_1,\ldots,m_k)$ if 
there exists a parking configuration of the $M$ cars that satisfies the following preferences for all drivers:
\begin{center}
\begin{tabular}{l}the driver of the $j$th car having colour $i>1$ asks that there\\ be {\red{at least}} $p^{(i)}_j$ cars of colours $\{1,\ldots ,i-1\}$ parked before him.
\end{tabular}
\end{center}
Drivers of cars having colour 1 have no preferences with regard to other coloured cars, which is why we only list their number $m_1$ in $P$.
\end{definition}

\begin{example}
The sequence $P=(4;(2,1,0,4,2),(8,2,1,2),(4,10,8))$ is a 4-tiered parking function of order $(4,5,4,3)$. 
This is realised by the following parking configuration where the leftmost entry represents the first parking spot, and the colour of the parked car is indicated.
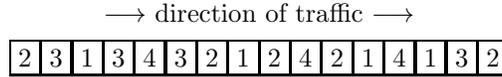
\begin{center}
$\longrightarrow$ direction of traffic $\longrightarrow$ \ \\[0.5em]
\car{2}\car{3}\car{1}\car{3}\car{4}\car{3}\car{2}\car{1}\car{2}\car{4}\car{2}\car{1}\car{4}\car{1}\car{3}\car{2} 
\end{center}
\end{example}

\begin{example}
The sequence $P=(9;(8,2,9,4,6,9,7,8,2,7,9,4),(18,2,14,6,21,13,7,13,3))$ is a 3-tiered parking function of order $(9,12,9)$.
This is realised by the following parking configuration:
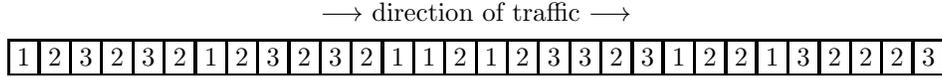
\begin{center}
$\longrightarrow$ direction of traffic $\longrightarrow$ \ \\[0.5em]
\car{1}\car{2}\car{3}\car{2}\car{3}\car{2}\car{1}\car{2}\car{3}\car{2}\car{3}\car{2}\car{1}\car{1}\car{2}\car{1}\car{2}\car{3}\car{3}\car{2}\car{3}\car{1}\car{2}\car{2}\car{1}\car{3}\car{2}\car{2}\car{2}\car{3}
\end{center}
\end{example}

We now connect these tired parking functions to the recurrent states of the sandpile model for the complete split graph.

\begin{theorem}
\begin{enumerate}
\item[]
\item[(a)] (Clique-sink case)
A configuration $c=(a_1,\ldots,a_{m-1},-;b_1,\ldots,b_n)$ on $\Sandpile(\Split_{m,n},v_m)$ is recurrent iff there exists a 3-tiered parking function of order $(m-1,n,m-1)$ with
$$P=((b_1,\ldots,b_n),(a_1+1,\ldots,a_{m-1}+1)).$$
\item[(b)] (Independent-sink case)
A configuration $c=(a_1,\ldots,a_{m};b_1,\ldots,b_{n-1},-)$ on $\Sandpile(\Split_{m,n},w_n)$  is recurrent 
iff there exists a 3-tiered parking function of order $(m,n-1,m)$ with
$$P=((b_1,\ldots,b_{n-1}),(a_1+1,\ldots,a_{m}+1)),$$
with the added restriction that there is no colour 2 car that is parked after all of the $n-1$ colour 1 cars.
\end{enumerate}
\end{theorem}

\begin{proof}
The equivalence is easily seen by associating a valid parking configuration with a reversal of the Motzkin word $\mw$ associated to a weakly decreasing recurrent configuration. 
Cars of colour $1$ correspond to $D$'s in a Motzkin word, while colour 2 corresponds to $H$'s and colour 3 corresponds to $U$'s.
\end{proof}

\begin{remark}
A slightly different formulation can be given in Definition~\ref{tpf} whereby cars of colour 1 are treated as empty parking spaces, and the number of empty places along with the number of cars of colour `less than' another colour be considered. 
\end{remark}

\begin{question}
How many 2-tiered parking functions are there? How many 3-tiered parking functions of order $(a,b,c)$ are there?
\end{question}

\section*{Acknowledgments}
My gratitude to the anonymous referees for their insightful suggestions on how to improve the presentation.

\end{document}